\documentclass[a4paper,12pt]{amsart}

\usepackage{hyperref}
\usepackage{url}
\usepackage{fullpage}
\usepackage{amssymb}
\usepackage{latexsym} 
\usepackage{amsfonts} 
\usepackage{amsmath}
\usepackage{eucal} 
\usepackage{bm} 
\usepackage{bbm} 
\usepackage{graphicx} 
\usepackage[english]{varioref} 
\usepackage[nice]{nicefrac} 
\usepackage[all]{xy}  
\usepackage{amsthm}

\newcommand{\Adm}{{\rm Adm}}

\newcommand{\Lie}{\text {\rm Lie}}

\def\Xo{\mathring{X}}
\def\Pio{\mathring{\Pi}}

\def\Om{\Omega}

\def\la{\lambda}
\def\tG{{\tilde G}}
\def\GG{{\mathcal G}}
\def\B{{\mathcal B}}
\def\O{{\bf O}}
\def\K{{\bf K}}
\def\st{{\star}}

\def\ge{\geqslant}
\def\le{\leqslant}
\def\a{\alpha}
\def\b{\beta}

\def\d{\delta}
\def\D{\Delta}

\def\t{\tau}
\def\th{\theta}

\def\l{\lambda}

\def\i{^{-1}}

\def\cf{\mathcal F}

\def\co{\mathcal O}
\def\cp{\mathcal P}

\def \trir {\triangleright}
\def \tril {\triangleleft}

\def\tS{\tilde S}

\def\etW{\widehat{W}}

\def\<{\langle} 
\def\>{\rangle}

\def\pt{{\rm pt}}

\def\Fl{\widetilde{Fl}}

\def\Gr{{\rm Gr}}
\def\Z{{\mathbb Z}}
\def\C{{\mathbb C}}
\def\R{{\mathbb R}}

\theoremstyle{plain}
\newtheorem{thm}{Theorem}[section] 
\newtheorem*{thm*}{Theorem} 
 \newtheorem{prop}[thm]{Proposition}
 \newtheorem{lem}[thm]{Lemma}
 \newtheorem{cor}[thm]{Corollary}

\theoremstyle{definition}

\theoremstyle{remark}

\newtheorem*{rmk*}{Remark}
\newtheorem*{claim*}{Claim}

\begin{document}

\author{Xuhua He}\address
 {Department of Mathematics\\ Hong Kong University of Science and Technology\\ Clear Water Bay\\ Kowloon, Hong Kong}
 \curraddr{Department of Mathematics\\ University of Maryland\\College Park\\MD 20742 USA.}
 \email{xuhuahe@math.umd.edu}
 \urladdr{http://www.math.umd.edu/~xuhuahe}
\author{Thomas Lam}\address
 {Department of Mathematics\\ University of Michigan\\ Ann Arbor\\ MI 48109 USA.}
 \date{\today}
 \email{tfylam@umich.edu}
 \urladdr{http://www.math.lsa.umich.edu/\~{ }tfylam}
\thanks{X.H. was partially supported by HKRGC grant 601409 and 602011.}
\thanks{T.L. was partially supported by NSF grants DMS-0968696 and DMS-1160726, and by a Sloan Fellowship.}
\title[]{Projected Richardson varieties and affine Schubert varieties}

\maketitle
\begin{abstract}
Let $G$ be a complex quasi-simple algebraic group and $G/P$ be a partial flag variety.  The projections of Richardson varieties from the full flag variety form a stratification of $G/P$.  We show that the closure partial order of projected Richardson varieties agrees with that of a subset of Schubert varieties in the affine flag variety of $G$.  Furthermore, we compare the torus-equivariant cohomology and $K$-theory classes of these two stratifications by pushing or pulling these classes to the affine Grassmannian.
Our work generalizes results of Knutson, Lam, and Speyer for the Grassmannian of  type $A$.
\end{abstract}

\section{Introduction}
Let $G$ be a complex quasi-simple algebraic group, $B, B^- \subset G$ be opposite Borel subgroups, and $T=B \cap B^-$ the maximal torus.  The flag variety $G/B$ has a stratification $G/B = \sqcup \Xo_w = \sqcup \Xo^{w}$ by Schubert cells $\Xo_w=B^- w B/B$ and opposite Schubert cells $\Xo^{w }=B w B/B$.  The intersections $\Xo_w^{v} = \Xo_w \cap \Xo^{v}$ are known as open Richardson varieties, and also form a stratification of $G/B$.  The closure of $\Xo^{v}_w$ is the (closed) Richardson variety $X^v_w=X^v \cap X_w$, where $X_w=\overline{B^- w B/B}$ is a Schubert variety and $X^v=\overline{B v B/B}$ is an opposite Schubert variety. 

Let $P \subset B$ be a fixed parabolic subgroup, and $\pi: G/B \to G/P$ denote the projection.  The \emph{open projected Richardson varieties} $\Pio_w^{v} = \pi(\Xo_w^{v})$ form a stratification of $G/P$ (for suitable $w$ and $v$).  Its closure $\Pi_w^v=\pi(X_w^v)$ is called a \emph{projected Richardson variety} and was studied by Lusztig \cite{L} and Rietsch \cite{R} in the context of total positivity, and by Goodearl and Yakimov \cite{GY} in the context of Poisson geometry.  Projected Richardson varieties enjoy many desirable geometric properties: Knutson, Lam, and Speyer \cite{KLS2} (see also Billey and Coskun \cite{BC}) showed that they are Cohen-Macaulay, normal, have rational singularities and are exactly the compatibly Frobenius split subvarieties of $G/P$ with respect to the standard splitting.

Our main results are combinatorial, cohomological, and $K$-theoretic comparisons between the projected Richardson varieties $\Pi_w^v$ and the affine Schubert varieties of the affine flag variety $\Fl$ of $G$.  These results generalize work of Knutson, Lam, and Speyer \cite{KLS} in the case that $G/P$ is a Grassmannian of type $A$.  The techniques of our proof differ significantly from those of \cite{KLS}.  In particular, the proof of the cohomological part of \cite{KLS} appears to only extend to cominuscule $G/P$.  A more geometric comparison in the Grassmannian case was performed by Snider \cite{S} who also recovered our $K$-theoretic comparison in the case of the Grassmannian.

On the combinatorial side, we compare two posets. One is obtained from the closure order of projected Richardson varieties, which we denote by $Q_J$. It was first studied by Rietsch \cite{R} and Goodearl and Yakimov \cite{GY}. The other one is the \emph{admissible subset} $\Adm(\l)$ associated to a dominant coweight $\l$, introduced by Kottwitz and Rapoport in \cite{KR}. It is a subset of the Iwahori-Weyl group $\etW$ and the Bruhat order on $\etW$ gives a partial order on $\Adm(\l)$. One important result in the study of Shimura varieties is that the special fiber of the local model is a union of finitely many opposite affine Schubert cells $I w I/I$ in the affine flag variety, where $w$ runs over the admissible set $\Adm(\l)$ for the Shimura coweight $\l$.  See \cite{PR} and \cite{Zhu}.

In this paper, we define an injection $\theta: Q_J \to \Adm(\l)$.  Our combinatorial theorem states that $\theta$ is order-reversing and the image is the $W \times W$-double coset of the translation element $t^{-\l}$. In the special case where $\l$ is a minuscule coweight, $\theta$ gives an order-reversing bijection between $Q_J$ and $\Adm(\l)$. The proof relies on properties of the \emph{Demazure}, or \emph{monoidal} product of Coxeter groups, studied for example by He, and He and Lu in \cite{H1,H2,HL}.  In Section \ref{s:applications}, we describe some applications of our result to the combinatorial properties of the poset $Q_J$, and to the enumeration of $\Adm(\l)$ for minuscule coweights $\l$. We also give a closed formula for the number of rational points of the special fiber of local model for ``fake'' unitary Shimura varieties. The order-reversing bijection between $Q_J$ and $\Adm(\l)$ also plays an important role in the proof of the normality and Cohen-Macaulayness of local models \cite{H13}. 

In fact, our combinatorial theorem naturally extends to the comparison of a larger partial order on $W \times W^J \supset Q^J$ with a $W \times W$-double coset in $\etW$.  This partial order on $W \times W^J$ arises as the closure partial order of a stratified space $Z_J$, studied by Lusztig \cite{L1}.  In Section \ref{ss:strata} we give maps between these stratified spaces which in part explains the combinatorial theorems.

Now let $\Gr = G(\K)/G(\O)$ denote the the affine Grassmannian of $G$, where $\K = \C((t))$ and $\O= \C[[t]]$.  Let $\Gr_\l \subset \Gr$ denote the closure of the $G(\O)$-orbit containing the torus-fixed point $t^{-\l}$ labeled by $-\l$. The dense open orbit $G(\O)t^{-\l}G(\O)/G(\O)$ is an affine bundle over $G/P$ and we let $p: G/P \hookrightarrow \Gr_\l$ denote the composition of the zero section with the open inclusion $G(\O)t^{-\l}G(\O)/G(\O) \subset \Gr_\l$.  Our $K$-theoretic theorem states that in $K_T(\Gr_\l)$ we have
\begin{equation}\label{E:cohom}
p_*([\co_{\Pi^x_y}]) = q^*(\psi^{\theta(x,y)})
\end{equation}
where $[\co_{\Pi^x_y}] \in K_T(G/P)$ and $\psi^{\theta(x,y)} \in K_T(\Fl)$ denote the $K$-theory classes of the structure sheaves of projected Richardson varieties and affine Schubert varieties respectively, and $q^*$ is induced by the composition of the inclusion $\Gr_\l \to \Gr$, with the maps $\Gr \simeq \Omega K \to LK \to LK/T_\R \simeq \Fl$.  Here $K \subset G$ denotes the maximal compact subgroup, $T_\R \subset T$ the compact torus, and $LK$ and $\Omega K$ are the free loop group and based loop group.  The same formula \eqref{E:cohom} holds for torus-equivariant cohomology classes.  The proof of the $K$-theoretic comparisons (Section \ref{sec:K}) relies on the study of equivariant localizations.  We utilize the machinery developed by Kostant and Kumar \cite{KKK} where equivariant localizations of Schubert classes in both finite and infinite-dimensional flag varieties are studied.

In Section \ref{ssec:Symm}, we use our $K$-theory comparison to prove a conjecture of Knutson, Lam, and Speyer \cite{KLS} stating that the affine stable Grothendieck polynomials of \cite{Lam,LSS} represent the classes of the structure sheaves of positroid varieties in the $K$-theory of the Grassmannian.  In Section \ref{ssec:QK}, we explain the implications, in the case of a cominuscule $G/P$, towards the comparison of the quantum $K$-theory of $G/P$ and the $K$-homology ring of the affine Grassmannian.

\medskip
{\bf Acknowledgements.}  The authors would like to thank Allen Knutson, Ulrich G\"ortz, Jiang-Hua Lu, Leonardo Mihalcea, David Speyer, and John Stembridge for related discussions. We also would like to thank the referee for careful reading and valuable suggestions. 

\section{Combinatorial comparison}\label{sec:combin}

%

\subsection{} Let $G$ be a complex connected quasi-simple algebraic group. Let $B, B^-$ be opposite Borel subgroup of $G$ and $T=B \cap B^-$ be a maximal torus. Let $Q$ be the coroot lattice and $P$ be the coweight lattice of $G$. We denote by $P^+$ the set of dominant coweights and $Q^+=Q \cap P^+$. Let $(\a_i)_{i \in S}$ be the set of simple roots determined by $(B, T)$. Let $R$ (resp. $R^+$, $R^-$) be the set of roots (resp. positive roots, negative roots). We denote by $W$ the Weyl group $N(T)/T$. For $i \in S$, we denote by $s_i$ the simple reflection corresponding to $i$.  For $\alpha \in R$, we let $r_\alpha$ denote the corresponding reflection.


Let $W_a=Q \rtimes W$ be the affine Weyl group and $\etW=P \rtimes W$ be the Iwahori-Weyl group (sometimes also called the extended affine Weyl group). It is known that $W_a$ is a normal subgroup of $\etW$ and is a Coxeter group with generators $s_i$ (for $i \in \tS=S \cup\{0\}$). Here $s_i$ (for $i \in S$) generates $W$ and $s_0=t^{-\th^\vee} s_\th$ is a simple affine reflection, where $\th$ is the largest positive root of $G$.  We emphasize that $\etW$, which serves as the key indexing set in the sequel, depends on $G$ and not just $R$.

Following \cite{IM65}, we define the length function on $\etW$ by 
\[\tag{a} \ell(t^\chi w)=\sum_{\a \in R^+, w \i(\a) \in R^+} |\<\chi, \a\>|+\sum_{\a \in R^+, w \i(\a) \in R^-} |\<\chi, \a\>+1|.
\] 

For any proper subset $J$ of $\tS$, let $W_J$ be the subgroup generated by $s_i$ for $i \in J$ and $w_J$ be the maximal element in $W_J$. We denote by $\etW^J$ (resp. ${}^J \etW$) the set of minimal length representatives in $\etW/W_J$ (resp. $W_J\backslash \etW$). For $J, K \subset \tS$, we simply write $\etW^J \cap {}^K \etW$ as ${}^K \etW^J$. If moreover, $J, K \subset S$, then we write $W^J$ for $W \cap \etW^J$, ${}^K W$ for $W \cap {}^K \etW^J$ and ${}^K W^J$ for $W \cap {}^K \etW^J$. For any $w \in W$, the coset $W_J w$ contains a unique minimal and a unique maximal element. We denote by $\min(W_J w)$ and $\max(W_J w)$ respectively. The elements $\min(w W_J)$ and $\max(w W_J)$ are defined in a similar way. 
%


Let $\Om$ be the subgroup of length-zero elements of $\etW$.  The Bruhat order on $W_a$ extends naturally to $\etW$. Namely, for $w_1, w_2 \in W_a$ and $\t_1, \t_2 \in \Om$, we define $\t_1w_1  \le \t_2 w_2$ if and only if $\t_1=\t_2$ and $w_1 \le w_2$ in $W_a$. 

\subsection{}\label{tr} Now we introduce three operations $\ast: \etW \times \etW \to \etW$, $\trir: \etW \times \etW \to \etW$ and $\tril: \etW \times \etW \to \etW$. Here $\ast$ is the \emph{Demazure}, or \emph{monoidal}, product and following \cite{KLS2} we call $\trir$ and $\tril$ the \emph{downwards Demazure products}. They were also considered in \cite{HL} and \cite{H2} and some properties were also discussed there. 

We describe $x \ast y$, $x \trir y$ and $x \tril y$ for $x, y \in \etW$ as follows. See \cite[Lemma 1.4]{H1}. 

(1) The subset $\{u v; u \le x, v \le y\}$ contains a unique maximal element, which we denote by $x \ast y$. Moreover, $x \ast y=u' y=x v'$ for some $u' \le x$ and $v' \le y$ and $\ell(x \ast y)=\ell(u')+\ell(y)=\ell(x)+\ell(v')$. 

(2) The subset $\{u y; u \le x\}$ contains a unique minimal element which we denote by $x \trir y$. Moreover, $x \trir y=u'' y$ for some $u'' \le x$ with $\ell(x \trir y)=\ell(y)-\ell(u'')$. 

(3) The subset $\{x v; v \le y\}$ contains a unique minimal element which we denote by $x \tril y$. Moreover, $x \tril y=x v''$ for some $v'' \le y$ with $\ell(x \tril y)=\ell(x)-\ell(v'')$. 

Now we list some properties of $\ast$, $\trir$ and $\tril$. 

(4) If $x' \le x$ and $y' \le y$, then $x' \ast y' \le x \ast y$. See \cite[Corollary 1]{H2}.

(5) If $x' \le x$, then $x' \tril y \le x \tril y$. See \cite[Lemma 2]{H2}. 

(6) $z \le x \ast y$ if and only if $z \tril y \i \le x$ if and only if $x \i \trir z \le y$. See \cite[Appendix]{HL}.

(7) If $J$ is a proper subset of $\tS$, then $\min(W_J x)=w_J \trir x$, $\min(x W_J)=x \tril w_J$, $\max(W_J x)=w_J \ast x$ and $\max(x W_J)=x \ast w_J$. 


\subsection{} 
In the rest of this section, we fix a dominant coweight $\l$. Set $J=\{i \in S; \<\l, \a_i\>=0\}$. Any element in $W t^{-\l} W \subset \etW$ can be written in a unique way as $y t^{-\l} x \i$ for $x \in W^J$ and $y \in W$. In this case, $\ell(y t^{-\l} x \i)=\ell(t^{-\l})+\ell(y)-\ell(x)$. The maximal element in $W t^{-\l} W$ is $w_S t^{-\l}$ and the minimal element is $t^{-\l} w_J w_S$. 

\begin{prop}\label{tl}
Let $x, x' \in W^J$ and $y, y' \in W$. Then the following conditions are equivalent:

(1) $y' t^{-\l} (x') \i \le y t^{-\l} x \i$;

(2) There exists $u \in W_J$ such that $y' u \le y$ and $x u \i \le x'$;

(3) There exists $v \in W_J$ such that $y' \le y v$ and $x v \le x'$. 
\end{prop}

\begin{proof} (1) $\Rightarrow$ (2): Since $\ell(y t^{-\l} x \i)=\ell(t^{-\l} x \i)+\ell(y)$, we have that $y t^{-\l} x \i=y \ast t^{-\l} x \i$. By \ref{tr} (6),  $y \i \trir (y' t^{-\l} (x') \i) \le t^{-\l} x \i$. In other words, there exists $z \le y$ such that $z \i y' t^{-\l} (x') \i \le t^{-\l} x \i$. Now we have that $$\max(z \i y' W_J) t^{-\l}=\max(z \i y' t^{-\l} (x') \i W) \le \max (t^{-\l} x \i W)=w_J t^{-\l}.$$

Therefore $\max(z \i y' W_J) \le w_J$ and $z \i y' \in W_J$. We denote $(y') \i z$ by $u$. Then $u \in W_J$, $y' u=z \le y$ and $u \i t^{-\l} (x') \i \le t^{-\l} x \i$. We have that 
\[u \i t^{-\l} (x') \i=(t^{-\l} w_J w_S) (w_S w_J u \i (x') \i) \quad \text{ and } \quad t^{-\l} x \i=(t^{-\l} w_J w_S) (w_S w_J x \i).
\]

Moreover, 
\begin{align*}
& \ell(u \i t^{-\l} (x') \i)=\ell(t^{-\l})+\ell(u)-\ell(x')  \\
&=\ell(t^{-\l} w_J w_S)+\ell(w_S w_J)+\ell(u)-\ell(x') \\ 
&=\ell(t^{-\l} w_J w_S)+\ell(w_S)-\ell(w_J)+\ell(u)-\ell(x') \\
&=\ell(t^{-\l} w_J w_S)+\ell(w_S)-\ell(w_J u \i)-\ell(x') \\
&=\ell(t^{-\l} w_J w_S)+\ell(w_S)-\ell(w_J u \i (x') \i) \\
&=\ell(t^{-\l} w_J w_S)+\ell(w_S w_J u \i (x') \i).
\end{align*}

Similarly, $\ell(t^{-\l} x\i)=\ell(t^{-\l} w_J w_S)+\ell(w_S w_J x \i)$.

From $u \i t^{-\l} (x') \i \le t^{-\l} x \i$ we deduce that $w_S w_J u \i (x') \i \le w_S w_J x \i$. Hence $w_J u \i (x') \i \ge w_J x \i$ and $x w_J \le x' u w_J$. By \ref{tr} (5), $$x u \i=(x w_J) \tril (w_J u \i) \le (x' u w_J) \tril (w_J u \i) \le x'.$$

(2) $\Rightarrow$ (1): We have that $y t^{-\l} x \i=y (t^{-\l} w_J w_S) (w_S w_J x \i)$ and $\ell(y t^{-\l} x\i)=\ell(y)+\ell(t^{-\l} w_J w_S)+\ell(w_S w_J x \i)$

Since $x u \i \le x'$, we have that $$x w_J=(x u \i) (u w_J) \le x' \ast (u w_J)=x' u w_J.$$ Thus $w_S w_J x \i \ge w_S w_J u \i (x') \i$. Also we have that $y' u \le y$. Therefore 
\begin{align*} 
y'  t^{-\l} (x') \i &=(y' u) (t^{-\l} w_J w_S) (w_S w_J u \i (x') \i) \le y (t^{-\l} w_J w_S) (w_S w_J x \i) \\
&=y t^{-\l} x \i.
\end{align*}

(2) $\Rightarrow$ (3): Since $y' u \le y$, by \ref{tr} (4) $y' \le y' u \ast u \i \le y \ast u \i$. In other words, there exists $v \le u \i$ such that $y' \le y v$. Notice that $u \in W_J$. Hence $v \in W_J$. Since $x \in W^J$, we also have that $x v \le x u \i \le x'$. 

(3) $\Rightarrow$ (2): Since $y' \le y v$, by \ref{tr} (5) $y' \tril v \i \le y v \tril v \i \le y$. In other words, there exists $u \le v \i$ such that $y' u \le y$. Notice that $v \in W_J$. Hence $u \in W_J$. Since $x \in W^J$, we also have that $x u \i \le x v \le x'$. 
\end{proof}

\subsection{} Define the partial order $\preceq$ on $W^J \times W$ as follows: 

$(x', y') \preceq (x, y)$ if and only if there exists $u \in W_J$ such that $x' u \le x$ and $y' u \ge y$. 

Define $Q_J=\{(x, y) \in W^J \times W; y \le x\}$. Then $(Q_J, \preceq)$ is a subposet of $(W^J \times W, \preceq)$.  We shall show in Appendix that $Q_J$ is the same poset as the one studied in \cite{R,GY}.

Following \cite{KR}, we introduce the admissible set as
$$\Adm(-w_S \l)=\{z \in \etW; z \le t^{-w \l} \text{ for some } w \in W\}.$$ Here $-w_S \l$ is the unique dominant coweight in the $W$-orbit of $-\l$. 

Now we have the following result. 

\begin{thm}\label{T:combin}\

(1) The map $$W^J \times W \to W t^{-\l} W, \qquad (x, y) \mapsto y t^{-\l} x \i$$ gives an order-preserving, graded, bijection between the poset $(W^J \times W, \preceq)$ and the poset $(W t^{-\l} W, \le^{op})$. Here $\le^{op}$ is the opposite Bruhat order on the Iwahori-Weyl group $\etW$. 

(2) Its restriction to $Q_J$ gives an order-preserving, graded, bijection between the posets $(Q_J, \preceq)$ and $(W t^{-\l} W \cap \Adm(-w_S \l), \le^{op})$. 
\end{thm}

\begin{proof} (1) is just a reformulation of the Proposition \ref{tl}. Now we prove (2).

If $(x, y) \in Q_J$, then $y \le x$. Hence $y t^{-\l} x \i \le x t^{-\l} x\i=t^{-x \l}$. So $y t^{-\l} x \i \in \Adm(-w_S \l)$. On the other hand, if $y t^{-\l} x \i \in \Adm(-w_S \l)$, then $y t^{-\l} x \i \le t^{-w \l}=w t^{-\l} w \i$ for some $w \in W^J$. Again by Proposition \ref{tl}, there exists $u \in W_J$ such that $y \le w u \le x$. Therefore $(x, y) \in Q_J$. 
\end{proof}

Theorem \ref{T:combin}(2) generalizes \cite[Theorem~3.16]{KLS}.

\section{Applications}
\label{s:applications}
\subsection{}
It is a classical result of Bj\"{o}rner and Wachs \cite{BW} that intervals in the Bruhat order of a Coxeter group satisfy nice combinatorial properties known as {\it thinness} and {\it shellability}.  Verma \cite{V} showed that the same intervals are {\it Eulerian}.  Dyer \cite{Dye} extended these results by showing that these intervals and their duals were more generally {\it EL-shellable}.  For the definitions of these combinatorial properties, we refer the reader to \cite{BB}; they will not play a role elsewhere in this paper.

Since Theorem \ref{T:combin} identifies each $Q_J$ with a convex subposet of (dual) affine Bruhat order we immediately obtain

\begin{cor}
The poset $Q_J$ is thin, Eulerian, and EL-shellable.
\end{cor}

This result was first established by Williams \cite{W}, who proved the more general result that the poset obtained from $Q_J$ by adjoining a maximal element is shellable.

\subsection{} 

Recall that a nonzero dominant coweight $\l$ is called {\it minuscule} if $\langle \l, \th \rangle=1$, where $\th \in R^+$ is the highest root.
Now we discuss the length-generating function $F_\l(q)$ of the admissible set $\Adm(-w_S \l)$ for a minuscule coweight $\l$. By definition, $$F_\l(q)=\sum_{w \in \Adm(-w_S \l)} q^{\ell(w)}.$$ This is the number of ${\mathbb F}_q$-rational points of the union of all opposite affine Schubert cells corresponding to the admissible set, where ${\mathbb F}_q$ is the finite field with $q$ elements. It is proved in \cite{PR} and \cite{Zhu} that this union is the special fiber of the local model of Shimura variety. 

Now by Theorem \ref{T:combin}, $F_\l(q)=\sum_{(x, y) \in Q_J} q^{\ell(y t^{-\l} x)}=\sum_{(x, y) \in Q_J} q^{\<\l, 2 \rho\>+\ell(y)-\ell(x)}$, where $\rho$ is the half sum of the positive roots of $G$. 

On the other hand, as we'll see in the appendix, $(Q_J, \preceq)$ is combinatorially equivalent to the poset of totally nonnegative cells in the cominuscule flag variety $G/P_J$ (that is, a partial flag variety $G/P_J$ where $J=\{i \in S; \<\l, \a_i\>=0\}$ for a minuscule coweight $\l$). The dimension of the cell corresponding to $(x, y) \in Q_J$ is $\ell(x)-\ell(y)$. Let $$A_J(q)=\sum_{(x, y) \in Q_J} q^{\ell(x)-\ell(y)}$$ be the rank generating function of totally nonnegative cells in $G/P_J$. Then we have that 
\begin{equation}\label{E:FA}
F_\l(q)=q^{\<\l, 2 \rho\>} A_J(q \i).
\end{equation}

In particular, the cardinality of $\Adm(-w_S \l)$ is $F_\l(1)=A_J(1)$.   When $G$ is of classical type, the numbers $A_J(1)$ and in some cases also the generating function $A_J(q)$ have been calculated:

\subsubsection{Type A}
Let $A_{k,n}(q) = A_J(q)$ where $G/P_J$ is the Grassmannian $\Gr(k,n)$ of $k$-planes in $n$-space, and similarly define $F_{k,n}(q)$.  Postnikov \cite{Pos} calculated $A_{k,n}(1)$ and Williams \cite{W2} established the formula
$$
A_{k,n}(q) = q^{-k^2} \sum_{i=0}^{k-1} (-1)^i \binom{n}{i} (q^{ki}[k-i]^i[k-i+1]^{n-i}-q^{(k+1)i}[k-i-1]^i[k-i]^{n-i})
$$
which by \eqref{E:FA} gives
$$
F_{k,n}(q) = \sum_{i=0}^{k-1} (-1)^i \binom{n}{i} (q^{i(n-k+1)}[k-i]^i[k-i+1]^{n-i}-q^{n+ni-ki}[k-i-1]^i[k-i]^{n-i}).
$$ Here $[i]=1+q+\cdots+q^{i-1}$ denotes the $q$-analog of $i$. 

In particular, 
$$
F_{k, n}(1)=\sum_{i=0}^{k-1} (-1)^i \binom{n}{i} ((k-i)^i (k-i+1)^{n-i}-(k-i-1)^i (k-i)^{n-i}).
$$

The formulas for $F_{1, n}(1)$ and $F_{2, n}(1)$ was first established by Haines in \cite[Proposition 8.2 (1) \& (2)]{Ha}. 

\subsubsection{Type B}
Let $F_{B_n}(q)$ denote $F_\lambda(q)$ for $\lambda = \omega_1^\vee$ the unique minuscule coweight when $G$ is adjoint of type $B$.  Similarly define $A_{B_n}(q)$.
\begin{prop}\label{P:B}
We have
$$
\sum_{n \geq 0} F_{B_n}(q)x^n = \frac{1+(-q-3q^2)x+(-q+5q^3+4q^4)x^2+q^4(-2-5q-3q^2)x^3+q^6[2]^2x^4}{(1-q^2x)(1-(q+q^2)x)(1 - [2]^2x + q^3[2]x^2)}.
$$
\end{prop}
\begin{proof}
Using the combinatorial description in \cite[Section 9]{LW}, we have the recursion
$$
A_{B_{n+1}}(q) =1+ (1+q)A_{B_{n}}(q) + \hat b_{n+1}(q)
$$
for $n \geq 1$, where as in \cite[Proposition 11.1]{LW}, $\hat b_{n}(q) = \sum_{(w_Sw_J,y) \in Q_J} q^{\ell(w_Sw_J)-\ell(y)}$.  This gives
$$
\sum_{n \geq 0} A_{B_{n}}(q)x^n = \frac{\hat b(x,q) -(1+q)x + \frac{x}{1-x}}{1-(1+q)x}
$$ 
where $\hat b(x,q) = \sum_{n \geq 0} \hat b_n(q) x^n$, and we have used the initial conditions $A_{B_0}(q) = 1$ and $A_{B_1}(q) = 2+q$.  Substituting the generating function for $\hat b_n(q)$ given in \cite[Proposition 11.1]{LW}, and using \eqref{E:FA} gives the stated result.
\end{proof}

\subsubsection{Type C}
Let $F_{C_n}(q)$ denote $F_\lambda(q)$ where $\lambda = \omega_n^\vee$ is the unique minuscule coweight when $G$ is adjoint of type $C$.  Haines \cite[Proposition 8.2 (3)]{Ha} showed that $F_{C_n}(1)=\sum_{i=0}^n 2^{n-i} n!/i!$, which is the greatest integer less than $2^n n! \sqrt{e}$.  This calculation was also done by Lam and Williams \cite[Proposition 11.3]{LW} where it is shown that $F_{C_n}(1)$ satisfies the recurrence $F_{C_0}(1) = 1$ and $F_{C_{n+1}}(1) = 2(n+1)F_{C_n}(1)+1$.

\subsubsection{Type D}
Let $F_{D_n}(q)$ denote $F_\lambda(q)$ where $\lambda = \omega_1^\vee$ is the minuscule coweight for $G$ simple of type $D$, such that $G/P_J$ is an even dimensional quadric.  Similarly define $A_{D_n}(q)$. 

\begin{prop}
We have
$$
\sum_{n \geq 0} F_{D_n}(q)x^n = \frac{1}{(1-q^2x)(1-(q+q^2)x)(1 - [2]^2x + q^3[2]x^2)} \times 
$$
{\footnotesize
$$
1-(q+3q^2)x-(q^2-q^3-4q^4)x^2-(2q+3q^2-2q^3-8q^4-2q^5+3q^6)x^3
+(2q^3+3q^4-3q^5-9q^6-4q^7+q^8)x^4
-(q^6-3q^8-2q^9)x^5
$$
}
\end{prop}

\begin{proof}
Using the combinatorial description in \cite[Section 9]{LW}, we have the recursion
$$
A_{D_{n+1}}(q) =1+ (1+q)A_{D_{n}}(q) + \hat d_{n+1}(q)
$$
for $n \geq 2$, where as in \cite[Proposition 11.2]{LW}, $\hat d_{n}(q) = \sum_{(w_Sw_J,y) \in Q_J} q^{\ell(w_Sw_J)-\ell(y)}$.  Declaring the initial conditions $A_{D_0}(q) = 1$, $A_{D_1}(q) = 2+q$, and $A_{D_2}(q) = 4+4q+q^2$ and proceeding as in the proof of Proposition \ref{P:B}, we obtain the stated result.
\end{proof}

\subsubsection{Type D}
The other minuscule coweights for type $D$ give the even orthogonal Grassmannians.  
The authors do not know of a calculation of $F_\lambda(q)$ in this case.  Part of the enumeration is done in \cite[Theorem 11.11]{LW}.

\section{Geometric comparison}

\subsection{} 
\label{ss:strata}
In this section, we explain some geometry behind the combinatorial comparison. Here we consider three stratified spaces. 

Let $J \subset S$ and $\lambda \in P^+$ be related by $J = \{i \in S; \langle \lambda,\alpha_i \rangle = 0\}$. Let $P_J$ be the standard parabolic subgroup of type $J$ and $L_J$ the standard Levi subgroup. Let $U_{P_J}$ be the unipotent radical of $P_J$. Let $\cp_J$ be the variety of parabolic subgroups conjugate to $P_J$. Then it is known that $\cp_J \cong G/P_J$. 

The first stratified space we consider is the partial flag variety $\cp_J \cong G/P_J$. By \cite{L} and \cite{R}, $G/P_J=\sqcup_{(x, y) \in Q_J} \Pio^{x}_y$ where $\Pio^{x}_y= \pi(\Xo^{x}_y)$ are the open projected Richardson varieties.  For any $(x, y) \in Q_J$, the closure of $\Pio^{x}_y$ is the union of $\Pio^{x'}_{y'}$, where $(x', y')$ runs over elements in $Q_J$ such that $(x', y') \preceq (x, y)$.

The second stratified space we consider is the variety $Z_J$ introduced by Lusztig in \cite{L1}. By definition, $$Z_J=(G \times G)/R_J,$$ where $R_J=\{(l u, l u'); l \in L_J, u, u' \in U_{P_J}\}$. 

For $(x, y) \in W^J \times W$, we define \begin{gather*} [J, x, y]^{+, -}=(B \times B^-) (x, y) R_J/R_J \subset Z_J.\end{gather*}

By \cite[2.2 \& 2.4]{HL}, $Z_J=\sqcup_{(x, y) \in W^J \times W} [J, x, y]^{+, -}$ and the closure of $[J, x, y]^{+, -}$ in $Z_J$ is the union of $[J, x', y']^{+, -}$, where $(x', y')$ runs over elements in $W^J \times W$ such that $(x', y') \preceq (x, y)$. 

The third stratified space is contained in the loop group $G(\K)$. Let $\O=\C[[t]]$ and $\O^-=\C[t \i]$. Let $I$ be the inverse image of $B$ under the map $p^+: G(\O) \to G$ by sending $t$ to $0$ and $I^-$ be the inverse image of $B^-$ under the map $p^-: G(\O^-) \to G$ by sending $t \i$ to $0$. Then we have that $G(\O^-) t^{-\lambda} G(\O)=\sqcup_{w \in W t^{-\l} W} I^- w I$. The closure of $I^- w I$ in $G(\O^-) t^{-\lambda} G(\O)$ is the union of $I^- w' I$, where $w'$ runs over elements in $W t^{-\l} W$ such that $w \le w'$ for the Bruhat order on $\etW$. 
 
We define $f: G(\O^-) t^{-\l} G(\O) \to Z_J$ as $f (g t^{-\l} (g') \i)=(p^+(g'), p^-(g)) R_J/R_J$ for $g \in G(\O^-)$ and $g' \in G(\O)$. Note that there is more than one way to write an element in 
$G(\O^-) t^{-\l} G(\O)$ as $g t^{-\l} (g') \i$ for $g \in G(\O^-)$ and $g' \in G(\O)$. Thus we need to check that the map $f$ is well-defined. 

\begin{lem}\label{lem:f}
The map $f$ is well-defined. 
\end{lem}

\begin{proof}
Let $I_1=\ker(p^+)$ and $I_1^-=\ker(p^-)$. Let $g, g' \in G$. Suppose that $g' t^{-\l} g \i \subset I^-_1 t^{-\l} I_1$, we'll show that $(g, g') \in R_J$. 

We have that $\emptyset \neq t^{\l} I_1^- g' t^{-\l} \cap I_1 g \subset t^{\l} G(\O^-) t^{-\l} \cap G(\O)$. We shall study this intersection in more detail. We show that 

(a) $G(\O^-) t^{-\l} I=\cup_{w \in W} I^- w t^{-\l} I$. 

Since $I^- w \subset G(\O^-)$ for all $w \in W$, $\cup_{w \in W} I^- w t^{-\l} I \subset G(\O^-) t^{-\l} I$. On the other hand, for any $i \in S$ and $w \in W$, $s_i I^- w t^{-\l} \subset  I^- s_i w t^{-\l} I \cup I^- w t^{-\l} I$. Hence $s_i \cup_{w \in W} I^- w t^{-\l} I \subset \cup_{w \in W} I^- w t^{-\l} I$. Since $G(\O^-)$ is generated by $I^-$ and $s_i$ for $i \in S$, we have that $G(\O^-) \cup_{w \in W} I^- w t^{-\l} I=\cup_{w \in W} I^- w t^{-\l} I$. In particular, $G(\O^-) t^{-\l} I=\cup_{w \in W} I^- w t^{-\l} I$. (a) is proved. 

Similarly, 

(b) For any $w \in W_J$, $I^- t^{-\l} I w I \subset \cup_{v \in W_J} I^- v t^{-\l} I$. 

Now for any $v \in W_J$ and $w \in {}^J W$, we have that $\ell(v t^{-\l} w)=\ell(v t^{-\l})-\ell(w)$. Hence $v t^{-\l} (I \cap w I^- w \i) \subset I^- v t^{-\l}$ and $I^- v t^{-\l} I w I=I^- v t^{-\l} (I \cap w I^- w \i) w I=I^- v t^{-\l} w I$. 

Now suppose that $g \in B w B$ for some $w \in W$. Then we may write $w$ as $w=x y$ for $x \in W_J$ and $y \in {}^J W$. Then applying (a) and (b) we deduce that $t^{-\l} I_1 g \subset \cup_{v \in W} I^- v t^{-\l} y I$ and $I^- g' t^{-\l} \subset \cup_{v \in W} I^- v t^{-\l} I$. Since $t^{-\l} I_1 g \cap I^- g' t^{-\l} \neq \emptyset$, by the disjointness of the Birkhoff factorization (see \cite[Theorem 5.23(g)]{Kum}) we have $y=1$ and $g \in P_J$. 

Assume that $g=u l$ with $u \in U_{P_J}$ and $l \in L_J$. Then $t^{-\l} g \i t^{\l} \subset l \i I_1^-$ and $t^{\l} I_1^- g' t^{-\l} g^{-1} \subset t^{\l} I_1^- g' l \i I_1^- t^{-\l}=t^{\l} I_1^- g' l \i t^{-\l}$, where for the last equality we use the fact that $G$ normalizes $I_1$.  Hence $t^{\l} I^-_1 g' l \i t^{-\l} \cap I_1 \neq \emptyset$. Now it follows from \cite[5.2.3 (11)]{Kum} that $I_1=(I_1 \cap t^{\l} I^-  t^{-\l})(I_1 \cap t^{\l} I t^{-\l})$. Since $U$ normalizes $I_1$, comparing Lie algebras and using the fact that $I_1$ is connected we obtain $I_1 \cap t^{\l} I  t^{-\l}=(I_1 \cap t^{\l} I_1  t^{-\l}) (I_1 \cap t^{\l} U  t^{-\l})$ and similarly $I_1 \cap t^{\l} I^-  t^{-\l}=(I_1 \cap t^{\l} I^-_1  t^{-\l}) (I_1 \cap t^{\l} U^-  t^{-\l})$.

It is easy to see that $t^{\l} U^- t^{-\l} \cap I_1=\{1\}$ and $t^{\l} U t^{-\l} \cap I_1=t^{\l} U_{P_J} t^{-\l}$. Thus $I_1=(I_1 \cap t^{\l} I^-_1  t^{-\l}) (I_1 \cap t^{\l} I_1  t^{-\l})  t^{\l} U_{P_J} t^{-\l}$.  Hence $g' l \i \in U_{P_J}$ and $(g, g') \in R_J$. The Lemma is proved. 
\end{proof}

\subsection{} The group $G(\O^-) \times G(\O)$ acts transitively on $G(\O^-) t^{-\l} G(\O)$.  It also acts transitively on $Z_J$ via the action $(g, g') \cdot z=(p^+(g'), p^-(g)) z$. The map $f: G(\O^-) t^{-\l} G(\O) \to Z_J$ is $G(\O^-) \times G(\O)$-equivariant. Thus all the fibers are isomorphic. Now we give an explicit description of the fiber over $R_J/R_J$. 

By Lemma \ref{lem:f}, it is $$\{I^-_1 U_{P_J} l t^{-\l} l \i U_{P_J} I_1; l \in L_J\}=I^-_1 U_{P_J} t^{-\l} U_{P_J} I_1,$$ where $I_1=\ker(p^+)$ and $I^-_1=\ker(p^-)$. Since $t^{-\l} U_{P_J} t^{\l} \subset I_1^-$ and $t^{\l} U_{P_J} t^{-\l} \subset I_1$, we have that \begin{align*} I^-_1 U_{P_J} t^{-\l} U_{P_J} I_1 &=I^-_1 t^{-\l} (t^{\l} U_{P_J} t^{-\l}) U_{P_J} I_1 \subset I^-_1 t^{-\l} I_1 U_{P_J} I_1 \\&=I^-_1 t^{-\l} U_{P_J} I_1=I^-_1 (t^{-\l} U_{P_J} t^{\l}) t^{-\l} I_1 \\ & \subset I^-_1 t^{-\l} I_1. \end{align*}

On the other hand, $I^-_1 t^{-\l} I_1 \subset I^-_1 U_{P_J} t^{-\l} U_{P_J} I_1$. Therefore the fiber over $R_J/R_J$ is $$I^-_1 U_{P_J} t^{-\l} U_{P_J} I_1=I^-_1 t^{-\l} I_1 \cong I^-_1 \times I_1/(I^-_1 \cap t^{-\l} I_1 t^{\l})$$ and is an infinite dimensional affine space. 

\subsection{} Let $\D: G \to G \times G$ be the diagonal embedding. Since $\D(P_J) \subset R_J$, there is a unique map $\iota: G/P_J \to Z_J$ such that the following diagram commutes 
\[\xymatrix{ G \ar[rr]^-{\D} \ar[d] & & G \times G \ar[d] \\ G/P_J \ar[rr]^-{\iota} & & Z_J.}\]

We have the following diagram which relates the three stratified spaces
\[\xymatrix{G/P_J \ar[r]^-{\iota} & Z_J & G(\O^-) t^{-\l} G(\O) \ar[l]_-f}.\]

This diagram is compatible with the respective stratifications: for $(x, y) \in W^J \times W$, $f(I^- y t^{-\l} x \i I)=(p^+(I) x, p^-(I^-) y) R_J/R_J=[J, x, y]^{+, -}$, agreeing with Theorem \ref{T:combin}(1).  The following proposition shows that the map $\iota$ preserves the stratifications on $G/P_J$ and $Z_J$, agreeing with Theorem \ref{T:combin}(2). 

\begin{prop}
For $(x, y) \in W^J \times W$, $\iota(G/P_J) \cap [J, x, y]^{+, -} \neq \emptyset$ if and only if $(x, y) \in Q_J$. In this case, $\iota(G/P_J)$ and $[J, x, y]^{+, -}$ intersect transversally and the intersection is $\iota(\Pio^{x}_y)$. 
\end{prop}

\begin{proof}
If $g \in B x B \cap B^- y B$, then $(g, g) R_J/R_J \in [J, x, y]^{+, -}$. Thus $\iota(\Pio^{x}_y)=\iota(\pi(\Xo^{x}_y)) \subset \iota(G/P_J) \cap [J, x, y]^{+, -}$, where $\pi: G/B \to G/P_J$ is the projection map. Since $\iota(G/P_J)=\sqcup_{(x, y) \in W^J \times W} \iota(G/P_J) \cap [J, x, y]^{+, -}$ and $G/P_J=\sqcup_{(x, y) \in Q_J} \Pio^{x}_y$, we have that $$ \iota(G/P_J) \cap [J, x, y]^{+, -}=\begin{cases} \iota(\Pio^{x}_y), & \text{ if } (x, y) \in Q_J; \\ \emptyset, & \text{ otherwise}. \end{cases}$$

Since $\iota(G/P_J)$ is a $\D(G)$-orbit on $Z_J$ and $[J, x, y]^{+, -}$ is a $B^- \times B$-orbit on $Z_J$ and $\Lie(\D(G))+\Lie(B^- \times B)=\Lie(G \times G)$, by \cite[Corollary 1.5]{Ri}, the intersection is transversal. 
\end{proof}

\section{$K$-theory comparison}
\label{sec:K}
\subsection{}
\label{ss:Kum}
In this subsection, let $\GG/\B$ be a Kac-Moody flag variety \cite{Kum}.  This is an ind-finite ind-scheme with a stratification by finite-dimensional Schubert varieties.  Let $W$ denote the Kac-Moody Weyl group with positive (resp. negative) roots $R^+$ (resp. $R^-$).  We consider $K$-cohomology with integer coefficients.  Kostant and Kumar \cite{KKK} constructed the Schubert basis of the torus-equivariant $K$-cohomology $K_T(\GG/\B)$, where $T \subset \GG$ is the maximal torus of the Kac-Moody group.   Let $\{\psi^v \mid v \in W\}$ denote the torus-equivariant Schubert basis $\psi^v \in K_T(\GG/\B)$ constructed by Kostant and Kumar \cite{KKK}.  We shall follow the notations of \cite{LSS}, which differ slightly from \cite{KKK}.  For the precise interpretation of $\psi^v$ as the class of a structure sheaf of a Schubert variety in the (possibly infinite-dimensional) $\GG/\B$ we refer the reader to \cite{LSS}.

For $v,w \in W$, we let $e_{v,w} = \psi^v(w):= \psi^v|_w \in K_T(\pt)$ denote the equivariant localization at the $T$-fixed point $v \in \GG/\B$.  A $K$-cohomology class $\psi \in K_T(\GG/\B)$ is completely determined by its equivariant localizations. We review certain facts concerning $e_{v,w}$.

If $W$ is a finite Weyl group, we denote by $w \mapsto w^\st$ the conjugation action $w \mapsto w_S w w_S$ by the longest element $w_S$.  The following result is \cite[Proposition 2.10]{LSS}.


\begin{thm} \label{thm:Kfixedpt}
Let $v, w \in W$ and $w=s_{i_1} \cdots s_{i_p}$ be a reduced expression. For $1 \le j \le p$, set $\b_j=s_{i_1} \cdots s_{i_{j-1}} \a_{i_j}$. Then $$e_{v, w}=\sum (-1)^{p-m} (1-e^{\b_{j_1}}) \cdots (1-e^{\b_{j_m}}),$$ where the summation runs over all those $1 \le j_1<\cdots<j_m \le p$ such that $v=s_{i_{j_1}} \ast \cdots \ast s_{i_{j_m}}$. 
\end{thm}

Define $E$ to be the (infinite) matrix $E = \left(e_{v,w}\right)$, and set $C = (E^{-1})^T$.  Define the matrix $C'$ by $c_{u^{-1}, v \i}=v \i w_S c'_{w_S v, w_S u}$.  Let $M$ be the Bruhat order matrix given by $m_{v,w} = 1$ if $v \geq w$, and $m_{v,w} = 0$ otherwise.  The following result is a variant of \cite[Proposition 4.16]{KKK}.

\begin{prop}\label{P:EDBM}
We have $E^T = DC'M$, where $D$ is the scalar matrix with value $\prod_{\alpha \in R^+}(1-e^\alpha)$.
\end{prop}

\begin{proof}
Define $E_{KK} = (M^T)^{-1}E$ and $C_{KK} = (E_{KK}^{-1})^T$.  Let $E'_{KK}$  be the ``$E$''-matrix of \cite{KKK}; then by \cite[Appendix A]{LSS} we have $(e'_{KK})_{v,w} = (e_{KK})_{v^{-1},w^{-1}}$.  Note that $m_{v,w} = m_{v^{-1},w^{-1}} = m_{w_0w,w_0v}$.  Proposition 4.16 of \cite{KKK} gives $E_{KK}^T = D C'_{KK} M^{-1}$, where $(c'_{KK})_{v,u} = v (c_{KK})_{u^{-1}w_S,v^{-1}w_S}$.  Then 
$$c'_{v,w} = v c_{w^{-1}w_S,v^{-1}w_S} = \sum_{u^{-1}w_S} n_{w^{-1}w_S,u^{-1}w_S} v(c_{KK})_{u^{-1}w_S,v^{-1}w_S} = \sum_{u} n_{u,v} (c'_{KK})_{v,u},$$
where $N = M^{-1}$.  It follows that $E^T = (M^T E_{KK})^T = E_{KK}^T M = DC'_{KK} = DC' M$.
\end{proof}

Now we prove some properties of $e_{v, w}$. 

\begin{lem}\label{L:K1} Let $W$ be any Kac-Moody Weyl group.
\begin{enumerate}
\item
We have
$$
e_{x,x} = \prod_{\alpha \in R^+ \cap x R^-}  (1-e^\alpha).
$$
\item
Suppose $W$ is a finite Weyl group.  Then
$$
e_{(x^\st)^{-1},(y^\st)^{-1} } = w_S y \i e_{x, y}.
$$
\item
Suppose $x,u,v \in W$ and $\ell(u v)=\ell(u)+\ell(v)$.  Then
$$
e_{x,uv} = \sum e_{u', u} (u e_{v',v}),
$$
where the summation runs over $u' , v' \in W$ such that $x=u' \ast v'$. 
\end{enumerate}
\end{lem}

\begin{proof}
(1) Let $x= s_{j_1} \cdots s_{j_p}$ be a reduced expression. Then $$e_{x, x}=(1-e^{\b_{j_1}})\cdots (1-e^{\b_{j_p}})=\Pi_{\a \in R^+ \cap x R^-} (1-e^\a).$$ 

(2) Let $y=s_{i_1} \cdots s_{i_p}$ be a reduced expression. Then $(y^\st)^{-1}= s^\st_{i_p} \cdots s^\st_{i_1}$ is also a reduced expression. Moreover, $x = s_{i_{j_1}} \ast \cdots \ast s_{i_{j_m}}$ if and only if $(x^\st) \i = s^\st_{i_{j_m}} \ast \cdots \ast s^\st_{i_{j_1}}$. For any $j$, we have that $w_S y \i (1-e^{\b_j})=w_S s_{i_p} \cdots s_{i_{j+1}} (1-e^{-\a_{i_j}})=s^\st_{i_p} \cdots s^\st_{i_{j+1}} (w_S (1-e^{-\a_{i_j}}))=s^\st_{i_p} \cdots s^\st_{i_{j+1}} (1-e^{\a_{i_j^\st}})$. Now \begin{align*} e_{(x^\st) \i, (y^\st) \i} &=\sum (-1)^{p-m} (w_S y \i (1-e^{\b_{j_1}})) \cdots (w_S y \i (1-e^{\b_{j_m}}))=w_S y \i e_{x, y}. \end{align*} Here the summation runs over all those $1 \le j_1<\cdots<j_m \le p$ such that $x = s_{i_{j_1}} \ast \cdots \ast s_{i_{j_m}}$. 

(3) Let $u=s_{i_1} \cdots s_{i_p}$ and $v=s_{i_{p+1}} \cdots s_{i_q}$ be reduced expressions. Let $1 \le j_1<\cdots<j_m \le p<j_{m+1}<\cdots<j_n \le p$ be such that $x = s_{i_{j_1}} \ast \cdots \ast s_{i_{j_n}}$. Then \begin{align*} & (1-e^{\b_{j_1}}) \cdots (1-e^{\b_{j_n}})=((1-e^{\b_{j_1}}) \cdots (1-e^{\b_{j_m}})) ((1-e^{\b_{j_{m+1}}}) \cdots (1-e^{\b_{j_n}})) \\ &=((1-e^{\b_{j_1}}) \cdots (1-e^{\b_{j_m}}))  u \bigl((s_{i_{p+1}} \cdots s_{i_{j_{m+1}-1}} (1-e^{\a_{i_{j_{m+1}}}})) \cdots (s_{i_{p+1}} \cdots s_{i_{j_n-1}} (1-e^{\a_{i_{j_n}}})) \bigr) \end{align*}
Now part (3) follows from Theorem \ref{thm:Kfixedpt}. 
\end{proof}


\begin{lem}\label{L1:fin}
Let $W$ be a finite Weyl group, and $u \leq u' \in W$.  Then
$$
\sum_{\substack{v, u'' \in W \\ u \leq v \leq u'' \leq u'}}e_{u^{-1},v^{-1}} (v^{-1} w_S e_{w_S u'', w_S v}) (v^{-1} w_S e_{w_S,w_S})^{-1} = \delta_{u,u'}.
$$
\end{lem}
\begin{proof}
By Proposition \ref{P:EDBM}, we have that
$$
\sum_{\substack{u'' \in W \\ v \leq u'' \leq u'}} e_{w_S u'', w_S v} e_{w_S,w_S}^{-1}=c'_{w_S v, w_S u'}.
$$
By definition, $c_{u'^{-1}, v \i}=v \i w_S c'_{w_S v, w_S u'}$.  Thus 
$$
\sum_{\substack{v, u'' \in W \\ u \leq v \leq u'' \leq u'}}e_{u^{-1},v^{-1}} (v^{-1} w_S e_{w_S u'', w_S v}) (v^{-1} w_S e_{w_S,w_S})^{-1} =\sum_{\substack{v \in W \\ u \leq v \leq u'}} e_{u \i, v \i} c_{u'^{-1}, v \i}=\d_{u, u'}.
$$
\end{proof}

\subsection{}
We shall apply the results of \ref{ss:Kum} in the case where $\GG/\B$ is the finite flag variety $G/B$, and in the case where $\GG/\B$ is the affine flag variety $\Fl = G(\K)/I$.  We return to the conventions of Section \ref{sec:combin}: $W$ denotes the finite Weyl group and $\etW$ denotes the Iwahori-Weyl group.  Note that the results of \ref{ss:Kum} from \cite{Kum} are stated for $W_a$ (that is, for the affine flag variety of the simply-connected $G$), but extend without change to the Iwahori-Weyl group $\etW$.  We shall also abuse notation in two ways.  First, if $x,y \in W$, we shall write $e_{x,y}$ without specifying whether we are considering equivariant localizations of affine or finite Schubert basis, since the two agree.  Secondly, if $x,y \in \etW$, then $\psi^x(y)$ normally takes values in $K_{\hat T}(\pt)$ for the affine torus $\hat T$.  In the following we still denote by $e_{x,y}$ the image of this value in $K_T(\pt)$ for the finite torus $T$.  That is, each affine root is considered as a root of the finite torus $T$ via the natural inclusion $T \hookrightarrow \hat T$.

Let $\Gr = G(\K)/G(\O)$ denote the affine Grassmannian of $G$.  Let $$\Gr_\l = \overline{G(\O)t^{-\l}G(\O)/G(\O)} \subset \Gr$$ be the closure of the $G(\O)$-orbit inside the affine Grassmannian containing the $T$-fixed point labeled by $t^{-\l}$.  Then $\Gr_\l = \sqcup_{\mu \le \l} G(\O)t^{-\mu}G(\O)/G(\O)$, where the union is over dominant coweights in dominance order.  The subscheme $\Gr_\l$ is not in general smooth, but the dense open orbit $G(\O)t^{-\l}G(\O)/G(\O)$ is smooth, being a (finite-dimensional) affine bundle over $G/P_J$, where $J=\{i \in S; \<\l, \a_i\>=0\}$ (see for example \cite[Section 2]{MV}).

\subsection{}
For general facts concerning equivariant $K$-theory, we refer the reader to \cite{CG}.  We now fix $W, J, \l$, and $y \in W$, $x,w \in W^J$.  The projected Richardson varieties $\Pi^x_y \subset G/P$ are labeled by $(x,y) \in Q_J$.  We denote by $[\co_{\Pi^x_y}] \in K_T(G/B)$ the torus-equivariant cohomology class, and by $[\co_{\Pi^x_y}]|_w \in K_T(\pt)$ the equivariant localization at a fixed point.  (Note that the $K$-cohomology and $K$-homology groups of $G/P$ are isomorphic, and we prefer to consider $[\co_{\Pi^x_y}]$ a class in $K$-cohomology.)  Write $\pi: G/B \to G/P$ for the natural projection.

\begin{prop}\label{P:PR}
We have
\begin{equation}\label{E1:6}
[\co_{\Pi^x_y}]|_w = \sum_{v \in W_J} e_{y,wv^{-1}} (wv^{-1}w_S e_{w_Sx^{-1},w_Svw^{-1}}) (wv^{-1} w_J e_{w_J,w_J} ) \i.
\end{equation}
\end{prop}
\begin{proof}
Recall that $X_y \subset G/B$ (resp. $X^x \subset G/B$) denotes the Schubert (resp. opposite Schubert) varieties.  Let $X^x_y = X^x \cap X_y \subset G/B$ be the Richardson variety.  Then in $K_T(G/B)$ we have the equality
$$
[\co_{X_w}][\co_{X^u}] = [\co_{X_w^u}]
$$
which follows from \cite[Lemma 1]{Bri}.  Thus,
$$
[\co_{X^x_y}]|_{u} = [\co_{X_y}]|_u [\co_{X^x}]|_u = e_{y,u} w_S e_{w_Sx,w_Su}
$$
since $X^x$ is obtained from $X_{w_Sx}$ by the action of $w_S$.  

By \cite[Theorem 4.5]{KLS2}, we have that $\pi_*\co_{X^x_y} = \co_{\Pi^x_y}$ and $R^i\pi_*\co_{X^x_y} = 0$ for $i > 0$.  Applying the equivariant pushforward $\pi_*:K_T(G/B) \to K_T(G/P)$ to $[\co_{X^x_y}]$ gives
$$
[\co_{\Pi^x_y}]|_w = (\pi_*[\co_{X^x_y}])|_w = \sum_{u \in wW_J} [\co_{X^x_y}]|_{u} e(\nu_u) \i
$$
where $e(\nu_u)$ denotes the $K$-theoretic equivariant Euler class of the tangent space $\nu_u$ at $u \in G/B$ to the fiber $\pi^{-1}(w)$.  This is a $K$-theoretic analogue of the Atiyah-Bott localization formula in equivariant cohomology; see for example \cite[Chapter 5]{CG}.  We calculate that
$$
e(\nu_u) = \prod_{\alpha \in R^- \, : \, ur_\alpha \in w W_J} u (1-e^\alpha) = u w_J e_{w_J,w_J}
$$
by Lemma \ref{L:K1}(1).  Finally, we apply Lemma \ref{L:K1}(2) to get $$w_S e_{w_Sx,w_Su}=uw_S e_{w_Sx^{-1},w_Su^{-1}}.$$
\end{proof}

\begin{lem}\label{L1:q}
For any $x,w \in W^J$ and $y \in W$, we have
$$
e_{yt^{-\l}x^{-1},wt^{-\l}w^{-1}} = \sum_{\substack{u \in W_J, y' \in W\\ y=y' \ast u \i }} e_{y',w} (w e_{t^{-\l}w_Jw_S,t^{-\l}w_Jw_S})(wt^{-\l}w_Jw_S e_{w_Sw_Ju^{-1}x^{-1},w_Sw_Jw^{-1}}).
$$
\end{lem}
\begin{proof}
We have $\ell(w(t^{-\l}w_Jw_S)(w_Sw_Jw^{-1})) = \ell(w)+\ell(t^{-\l}w_Jw_S)+\ell(w_Sw_Jw^{-1})$.  By Lemma \ref{L:K1}(3), 
$$
e_{y t^{-\l} x\i, w t^{-\l} w\i}=\sum e_{v,w} (we_{v',t^{-\l}w_Jw_S})(wt^{-\l}w_Jw_S e_{v'', w_Sw_Jw^{-1}}),
$$ 
where the summation runs over $v \leq w$, $v' \leq t^{-\l}w_Jw_S$ and $v'' \leq w_Sw_Jw^{-1}$ such that $v \ast v' \ast v'' = yt^{-\l}x$.  But $t^{-\l}w_Jw_S$ is minimal in $Wt^{-\l}W$, and since $v,v'' \in W$, we must have $v' = t^{-\l}w_Jw_S$.  If $v'' = w_Sw_Ju^{-1}x^{-1}$, then $v \ast (t^{-\l}w_Jw_S) \ast v''$ lies in $Wt^{-\l}u^{-1}x^{-1}$.  It follows that we must have $u^{-1} \in W_J$, and $(t^{-\l}w_Jw_S) \ast v'' = u^{-1}t^{-\l}x^{-1}$.  The result follows.
\end{proof}

\subsection{}\label{ssec:H}
The affine Grassmannian $\Gr$ is weak homotopy-equivalent to the based loop group $\Omega K$, where $K \subset G$ is a maximal compact subgroup.  The affine flag variety $\Fl$ is weak homotopy-equivalent to the quotient $LK/T_\R$ of the (unbased) loop group by the compact torus.  The torus-equivariant composition $\Omega K \to LK \to LK/T_\R$ induces a pullback map $r^*:K_T(\Fl) \to K_T(\Gr)$, and we refer the reader to \cite[Section 5]{HHH} and \cite{LSS} for a discussion of this.  The $T$-fixed points of $\Gr$ are labeled by the cosets of $\etW/W$.  Thus the pullback map can be described (\cite[Lemma 4.6]{LSS}) in terms of equivariant localizations by the formula
\begin{equation}
\label{eq:pullback}
r^*(\psi)(t^{-\l}W) = \psi(t^{-\l})
\end{equation} 
for $\psi \in K_T(\Fl)$ and a dominant coweight $\l$.  For our purposes, we can take this formula to be the definition of $r^*$; indeed, it is checked algebraically in \cite{LSS} that \eqref{eq:pullback} gives a well-defined map on the equivariant $K$-theories.  

Let $p': G/P \to G(\O)t^{-\l}G(\O)/G(\O)$ denote the zero-section of the affine bundle $G(\O)t^{-\l}G(\O)/G(\O) \to G/P$, and let $p: G/P \to \Gr_\l$ be the composition of $p'$ with the open inclusion $G(\O)t^{-\l}G(\O)/G(\O) \subset \Gr_\l$.  The map $p$ (and also $p'$) is a (torus-equivariant) closed embedding, and thus a proper map.  

We have a pushforward map $p_*:K_T(G/P) \to K_T(\Gr_\l)$, defined as follows.  Let $K^T(G/P)$ and $K^T(\Gr_\l)$ denote the Grothendieck groups of $T$-equivariant coherent sheaves on $G/P$ and $\Gr_\l$ respectively, as defined in \cite[Chapter 5]{CG}.  We have already remarked that $K^T(G/P) \simeq K_T(G/P)$.  We also have $K^T(\Gr_\l) \simeq K_T(\Gr_\l)$ by \cite[Proposition 5.5.6]{CG} applied to the Schubert stratification of $\Gr_\l$ (noting that the constructions of Kostant and Kumar \cite{KKK} are carried out in topological equivariant $K$-theory).
%
The pushforward map $p_*:K^T(G/P) \to K^T(\Gr_\l)$ is defined on the level of coherent sheaves by $p_*[\cf] = \sum_i (-1)^i[R^ip_*\cf]$, whenever $p$ is a proper map.  In fact $R^ip_* = 0$ for $i > 0$ since $p$ is a closed embedding in our situation.  We use the same notation $p_*$ to denote the map $K_T(G/P) \to K_T(\Gr_\l)$ obtained by composing with the isomorphisms $K^T(G/P) \simeq K_T(G/P)$ and $K^T(\Gr_\l) \simeq K_T(\Gr_\l)$.


Let $q^*: K_T(\Fl) \to K_T(\Gr_\l)$ be the composition of $r^*$ with the restriction $K_T(\Gr) \to K_T(\Gr_\l)$.  The following theorem generalizes \cite[Theorem 12.8]{KLS} in two ways: from the Grassmannian to all partial flag varieties $G/P$, and from cohomology to $K$-theory.

\begin{thm}\label{T:Kmain}
We have $p_*([\co_{\Pi^x_y}]) = q^*(\psi^{yt^{-\l}x^{-1}})$.
\end{thm}
\begin{proof}
We have $yt^{-\l}x^{-1} \le t^{\mu}$ only if $ \max(W yt^{-\l}x^{-1} W) \le \max(W t^{\mu} W)$, only if 
$\mu' \ge \l$, where $\mu'$ is the dominant coweight in the $W$-orbit of $\mu$.  Thus $q^*(\psi^{yt^{-\l}x^{-1}})$ is non-zero only on $T$- fixed points of the form $w t^{-\l} w^{-1} \in G(\O)t^{-\l}G(\O)/G(\O)\subset \Gr_\l$.  Since $p(G/P) \subset G(\O)t^{-\l}G(\O)/G(\O)$, the class $p_*([\co_{\Pi^x_y}])$ is also supported on the same $T$-fixed points.  A class in $K_T(\Gr_\l)$ is determined by its pullbacks to all the $T$-fixed points, and so it suffices to compare the two sides at each of the $T$-fixed points $w t^{-\l} w^{-1}$.


By Lemma \ref{L:K1}(3), we have
\begin{equation}\label{E1:3}
e_{y,wv^{-1}} =\sum_{\substack{u \leq v, y' \in W \\ y=y' \ast u \i}} e_{y',w}(w e_{u^{-1},v^{-1}}).
\end{equation}
and
\begin{align*}
e_{w_S x^{-1}, w_Svw^{-1}} &= e_{w_S x^{-1}, (vw_J)^\st w_S w_J w^{-1}} = \sum e_{(u''w_J)^\st,(vw_J)^\st}\left((vw_J)^\st e_{w_Sw_J u'^{-1}x^{-1},w_S w_J w^{-1}}\right),
\end{align*} where the summation runs over all $u', u'' \in W$ such that $(u'' w_J)^\st \ast (w_S w_J u'^{-1} x \i)=w_S x \i$. By definition, $w_S w_J u'^{-1} x\i \in W_{J^\st} w_S x \i=w_S W_J x \i$ and $u' \in W_J$. Then $w_S w_J u'^{-1} x\i=(w_J u'^{-1})^\st w_S x \i$, here $(w_J u'^{-1})^\st \in W_{J^\st}$ and $w_S x \i$ is the maximal element in $W_{J^\st} w_S x \i$. Hence $(u'' w_J)^\st \ast (w_S w_J u'^{-1} x \i)=w_S x \i$ if and only if $(u'' w_J)^\st \ge ((w_J u'^{-1})^\st) \i$, i.e., $u'' w_J \ge u' w_J$. This is equivalent to say that $u'' \le u'$. Hence 
\begin{equation}\label{E1:4}
e_{w_S x^{-1}, w_Svw^{-1}}=\sum_{v \le u'' \le u' \text{ in } W_J} e_{(u''w_J)^\st,(vw_J)^\st}\left((vw_J)^\st e_{w_Sw_J u'^{-1}x^{-1},w_S w_J w^{-1}}\right)
\end{equation}

Thus applying Lemma \ref{L:K1}(2),
\begin{equation}
v \i w_S e_{(u''w_J)^\st,(vw_J)^\st}=w_J e_{w_J u''^{-1}, w_J v \i}=v \i w_J e_{w_J u'', w_J v}
\end{equation}
and 
\begin{equation}\label{E1:5}
wv^{-1}w_S e_{w_S x^{-1},w_Svw^{-1}}  = \sum_{v \le u'' \le u' \text{ in } W_J} (wv^{-1} w_J e_{w_J u'',w_J v}) (ww_Jw_S e_{w_Sw_J u'^{-1}x^{-1},w_S w_J w^{-1}}).
\end{equation}

Substituting \eqref{E1:3} and \eqref{E1:5} into \eqref{E1:6}, we have that 
$$
[\co_{\Pi^x_y}]|_w =\sum e_{y', w} (w e_{u \i, v \i})(wv^{-1} w_J e_{w_J u'',w_J v}) (wv^{-1} w_J e_{w_J,w_J} ) \i (ww_Jw_S e_{w_Sw_J u'^{-1}x^{-1},w_S w_J w^{-1}}),
$$ where the summation is over $u \le v \le u'' \le u' \text{ in } W_J$ and $y' \in W$ such that $y=y' \ast u \i$.

Applying Lemma \ref{L1:fin} to
$$
\sum_{\substack{v, u'' \in W_J \\ u \leq v \leq u'' \leq u'}}(w e_{u^{-1},v^{-1}}) (wv^{-1} w_J e_{w_J u'', w_J v}) (wv^{-1} w_J e_{w_J,w_J})^{-1}
$$
gives
$$
[\co_{\Pi^x_y}]|_w = \sum_{\substack{u \in W_J, y' \in W\\ y=y' \ast u \i }} e_{y',w}(ww_Jw_S e_{w_Sw_Ju^{-1}x^{-1},w_Sw_Jw^{-1}}).
$$

To compute $p_*([\co_{\Pi^x_y}])|_{wt^{-\l}w^{-1}}$ we apply the formula \cite[Theorem 5.11.7]{CG} for the localization of a pushforward, in the form given in \cite[Equation (7)]{FS}.  Since $\Gr_\l$ is not smooth, to apply these results, we must first restrict to the smooth open subset $G(\O)t^{-\l}G(\O)/G(\O)$.  Let $j: G(\O)t^{-\l}G(\O)/G(\O) \hookrightarrow \Gr_\l$ denote the open inclusion.  Then for any class $\psi \in K_T(\Gr_\l)$, we have $\psi|_{wt^{-\l}w^{-1}} = (j^*(\psi))|_{wt^{-\l}w^{-1}}$ by composing pullbacks.  It thus suffices to calculate $(j^*(p_*([\co_{\Pi^x_y}]))|_{wt^{-\l}w^{-1}}$.  Since $p$ is a closed embedding, and $j$ is an open embedding, we have $j^*(p_*([\co_{\Pi^x_y}])) = p'_*([\co_{\Pi^x_y}])$.  Now, $p': G/P \to G(\O)t^{-\l}G(\O)/G(\O)$ is a torus-equivariant closed embedding of smooth varieties, and we apply \cite[Equation (7)]{FS}, as follows.

We have that $p'_*([\co_{\Pi^x_y}])|_{wt^{-\l}w^{-1}} = e(\nu_w)[\co_{\Pi^x_y}]|_w$, where $e(\nu_w)$ is a product of $(1-e^\beta)$ over the weights $\beta$ of the normal space
to $G/P \subset G(\O)t^{-\l}G(\O)/G(\O)$ at the point $wt^{-\l}w^{-1}$.  The factor $e(\nu_{w})$ is equal to the product of the weights of the $T$-invariant curves joining $wt^{-\l}w^{-1}$ to $T$-fixed points $z$ inside $\Gr_\l$ which are outside of $G/P$.  For $w = w_S w_J$, these are all $T$-fixed points of the form $z = r_\alpha t^{-\l w_J w_S} < t^{-\l} w_J w_S$.  The product of the $T$-weights is thus $e(\nu_{1})=e_{t^{-\l}w_Jw_S, t^{-\l}w_Jw_S}$.  A similar calculation gives $e(\nu_{w}) = we_{t^{-\l}w_Jw_S, t^{-\l}w_Jw_S}$.

Combining with Lemma \ref{L1:q}, we get
$$
p_*([\co_{\Pi^x_y}])|_{wt^{-\l}w^{-1}} = e(\nu_w)[\co_{\Pi^x_y}]|_w = q^*(\psi^{yt^{-\l}x^{-1}})|_{wt^{-\l}w^{-1}}.
$$
\end{proof}

\subsection{}
As claimed in the introduction, the analogue of Theorem \ref{T:Kmain} holds in equivariant cohomology, either by a similar but easier proof, or by looking at the ``lowest degree terms'' of the equivariant localizations.

\begin{thm}\label{T:Cmain}
We have $p_*([\Pi^x_y]) = q^*(\xi^{yt^{-\l}x^{-1}})$, where $[\Pi^x_y] \in H^*_T(G/P)$ denotes the equivariant cohomology class of a projected Richardson variety, and $\xi^{yt^{-\l}x^{-1}} \in H^*_T(\Fl)$ are the torus-equivariant cohomlogy classes of Schubert varieties, constructed by Kostant and Kumar \cite{KK}.
\end{thm}

\subsection{}\label{ssec:Symm}
For background material on the symmetric function notation used in this section, we refer the reader to \cite{Buch,LSS}.  The general strategy of this section is similar to \cite[Section 7]{KLS}.

In this section we let $G = PGL(n,\C)$ and $G/P$ be the Grassmannian $\Gr(k,n)$ of $k$-planes in $\C^n$.  In \cite{Buch}, Buch defined {\it stable Grothendieck polynomials} $G_\la(X) \in \hat \Lambda$ for each partition $\la$, lying in the graded completion $\hat\Lambda$ of the ring of symmetric functions.  Buch showed that the $K$-theory $K(\Gr(k,n))$ of the Grassmannian could be presented as $\Gamma/I_{k,n}$, where $\Gamma = \prod_\la \Z \cdot G_\la(X)$, and $I_{k,n}$ is the ideal spanned (as a direct product) by all $G_\la$ where $\la$ is not contained in a $k \times (n-k)$ rectangle.  (Buch considered the direct sum rather than product of the $\Z \cdot G_\la(X)$, but the quotient is the same.)

In \cite{Lam}, symmetric functions $\tG_w(X)$ called {\it affine stable Grothendieck polynomials} were defined for each element $w \in\etW$ of the affine Weyl group (in this case, the affine symmetric group).  Let $\Lambda^{(n)}$ be the quotient of the ring of symmetric functions by the ideal generated by all monomial symmetric functions $m_\la$, for $\la_1 \geq n$.  Let $\hat \Lambda^{(n)}$ be the graded completion of $\Lambda^{(n)}$.
Let $r^*:K(\Fl) \to K(\Gr)$ denote the pullback map in $K$-theory, as in Subsection \ref{ssec:H}.
In \cite{LSS}, it was shown that $K(\Gr) \simeq \hat \Lambda^{(n)}$ and that under this isomorphism one has
\begin{equation}
\label{E:LSS}
r^*(\psi^w) = \tG_w.
\end{equation}

The following result was conjectured in \cite[Conjecture 7.11]{KLS}.  

\begin{thm}
Let $\omega_k$ denote the $k$-th fundamental coweight. Under the isomorphism $\kappa: K(\Gr(k,n)) \simeq \Gamma/I_{k,n}$, we have 
$$
\kappa([\co_{\Pi^x_y}]) = \tG_{yt^{-\omega_k}x^{-1}}
$$
where the right hand side is considered as an element of the quotient $\Gamma/I_{k,n}$.
\end{thm}
\begin{proof}
When $\la$ is the fundamental coweight $\omega_k$ we have $G/P \simeq \Gr_\la \subset \Gr$ (see \cite[Section 7]{KLS}).
Combining (the non-equivariant image of) Theorem \ref{T:Kmain} with \eqref{E:LSS}, it thus remains to check that the inclusion $\iota: G/P \hookrightarrow \Gr$ induces the natural quotient map $\hat \Lambda^{(n)} \to \Gamma/I_{k,n}$.  

The ring $\hat \Lambda^{(n)}$ contains distinguished symmetric functions $G_{(m)}(X) = \tG_{s_{m-1}s_{m-2}\cdots s_0}(X)$ for $1 \leq m < n$.  The completion of the subring generated by the $G_{(m)}(X)$ is exactly $\hat \Lambda^{(n)}$.  The ring homomorphism $\iota^*:K(\Gr) \to K(\Gr(k,n))$ is compatible with graded completions, and is thus determined by the images of $G_{(m)}(X)$.  Now, in $K^*(\Gr)$, $\tG_{s_{m-1}s_{m-2}\cdots s_0}(X)$ represents the pullback $r^*(\psi^{s_{m-1}s_{m-2}\cdots s_0})$.  For $m \leq n-k$, modulo length-zero elements of $\etW$ (one has $\tG_v(X) = \tG_u(X)$ if $v$ and $u$ differ by a length-zero element), $s_{m-1}s_{m-2}\cdots s_0$ is the same as $s_{k+m-1} \cdots s_{k+1} s_k t^{-\omega_k}w_J w_S$.  But under Buch's isomorphism $K(\Gr(k,n)) \simeq \Gamma/I_{k,n}$, the opposite Schubert variety $\pi(X_{s_{k+m-1} \cdots s_{k+1} s_k}) = \Pi^{w_Jw_S}_{s_{k+m-1} \cdots s_{k+1} s_k}$ is represented by the symmetric function $G_{(m)}(X)$ as well \cite[Theorem 8.1]{Buch}.  Similarly, if $m > n-k$ one sees that $\iota^*$ sends $G_{(m)}(X)$ to 0.  Thus $\iota^*$ induces the natural map $\hat \Lambda^{(n)} \to \Gamma/I_{k,n}$.
\end{proof}

\subsection{}\label{ssec:QK}
In \cite{LSQH}, Lam and Shimozono, following work of Peterson, showed that the quantum cohomology rings $QH^*(G/P)$ of partial flag varieties could, after localization, be identified with a quotient of the homology $H_*(\Gr)$ of the affine Grassmannians.  In particular, the 3-point Gromov-Witten invariants of $G/P$ could be recovered from the homology Schubert structure constants of $H_*(\Gr)$.  

Let $G/P$ be a cominuscule flag variety.  In this section, we discuss the implications of Theorem \ref{T:Kmain} towards the comparison of the quantum $K$-theory $QK(G/P)$ of $G/P$ and $K$-homology $K_0(\Gr)$ of the affine Grassmannian.  We will work in the non-equivariant setting; the $T$-equivariant statements are analogous.  We now define four sets of integers.

\begin{enumerate}
\item
For $u,v,w \in \etW/W$, let $d^w_{uv} \in \Z$ denote the $K$-homology Schubert structure constants of $K_0(\Gr)$, defined in \cite[(5.3)]{LSS} (we will only consider the non-equivariant structure constants).  We remark that in \cite{LSS} only the affine Grassmannian $\Gr$ of a simply-connected simple algebraic group is considered, but the extension is straightforward; see for example \cite{LSQT}.
\item
For $u \in \etW$, and $y \in \etW^S$ a minimal length coset representative of $\etW/W$ we can consider the coefficient $k_y^u$ of the $K$-cohomology Schubert class $\psi_{\Gr}^y$ in $r^*(\psi_{\Fl}^u)$, where $r^*:K(\Fl) \to K(\Gr)$ denotes the pullback map in $K$-theory, as in Subsection \ref{ssec:H}.
\item
For a positroid variety $\Pi^u_v$ and $y \in W^J$, consider the coefficient $\pi^y_{(u,v)}$ of the (class of the) Schubert structure sheaf $[\co_{X_y}]$ in $[\co_{\Pi^u_v}] \in K(G/P)$.
\item
For $u,v,w \in W^J$, consider the $K$-theoretic Gromov-Witten invariant $I_d(u,v,w) = I_d(\co_{X_u},\co_{X^v},(\co_{X_w})^\vee)$; see for example \cite[Section 5]{BCMP}.  Here $\{[(\co_{X_w})^\vee]\}$ is the dual basis to $\{[\co_{X_w}]\}$ in $K(G/P)$.  The $K$-theoretic Gromov-Witten invariant is defined as the Euler characteristic of the product of the pullbacks of these structure sheaves to the moduli space $M_{d,3}(G/P)$ of three-point, genus zero, stable maps into $G/P$ with degree $d$.
\end{enumerate}

We now compare the four sets of integers.
\begin{enumerate}
\item
By \cite[(5.1), (5.4) and Theorem 5.4]{LSS}, 
$$
d_{uv}^w = \sum_{\substack{x \in \etW \\ x*v = w}} (-1)^{\ell(w) -\ell(v)-\ell(x)}k^x_u.
$$
Thus the $k^x_u$ determine the $d_{uv}^w$, and it is easy to see that (picking $v$ and $u$ appropriately) the $d_{uv}^w$ also determine the $k^x_u$.
\item
By the cominuscule assumption we have $G/P \simeq \Gr_\la \subset \Gr$, where $\la$ is the appropriate minuscule coweight.    Thus $p_*$ can be identified with the identity, and Theorem \ref{T:Kmain} states that $q^*(\psi^{yt^{-\la}x^{-1}}) =[\co_{\Pi^x_y}]$.  Now suppose that $x = w_Sw_J$ and $y \in W^J$.  Then $\Pi^x_y = \pi(X_y)$ is a usual Schubert variety in $G/P$, and since $yt^{-\la} w_Jw_S \in \etW^S$, we have $r^*(\psi^y_{\Fl}) = (\psi^y_{\Gr})$.  It follows that the coefficient $\pi^y_{(u,v)}$ is equal to $k_y^{vt^{-\la}u^{-1}}$. 
\item
In \cite{BCMP}, Buch, Chaput, Mihalcea, and Perrin studied the geometry of the Gromov-Witten varieties associated to cominuscule $G/P$.  An unpublished\footnote{Since this paper was written, a preprint with related results has appeared as \emph{Projected Gromov-Witten varieties in cominuscule spaces},
A. S. Buch, P.-E. Chaput, L. C. Mihalcea, and N. Perrin, \texttt{arXiv:1312.2468}.} consequence of their work, communicated to us by L. Mihalcea, is that
\begin{equation}\label{E:BCMP}
I_d(u,v,w) \text{ is equal to the coefficient of  a Schubert structure sheaf $[\co_{X_w}]$ in } [\co_{\Pi^x_y}]
\end{equation}
where $\Pi^x_y$ is a projected Richardson variety which depends on $d,u,v$.  For an explicit description of $\Pi^x_y$ in type $A$ see \cite[Section 8]{KLS}.  For the classical types the explicit description can presumably be recovered from \cite{BKT}, and for other cominuscule types see \cite{CMP}.  Thus the coefficients $\pi^y_{(x,z)}$ determine all the coefficients $I_d(u,v,w)$.
\end{enumerate}

\begin{cor}
Let $G/P$ be cominuscule.  Assuming \eqref{E:BCMP}, the $K$-homology Schubert structure constants determine the 3-point $K$-theoretic Gromov-Witten invariants of $G/P$.
\end{cor}

\section*{Appendix}
\label{app}
Here we prove that the poset $(Q_J, \preceq)$ is combinatorially equivalent to the poset introduced by Rietsch in \cite[Section 5]{R} and by Goodearl and Yakimov in \cite[Theorem 1.8]{GY}.

Following \cite{R}, we set $$Q'_J=\{(a, b, c) \in W^J_{\max} \times W_J \times W^J; a \le c b\}$$ and define the partial order $\le$ on $Q'_J$ as follows.

For $(a, b, c), (a', b', c') \in Q'_J$, define $$(a', b', c') \le (a, b, c)$$ if there exists $u'_1, u'_2 \in W_J$ with $u'_1 u'_2=b'$, $\ell(u'_1)+\ell(u'_2)=\ell(b')$ and $$a b \i \le a' (u'_2) \i \le c' u'_1 \le c.$$

Following \cite{GY}, we set $$\Om_J=\{(a, b) \in W^J_{\max}  \times W; a \le b\}$$ and define the partial order $\le$ on $\Om_J$ as follows.

For $(a, b), (a', b') \in \Om_J$, define $$(a', b') \le (a, b)$$ if there exists $z \in W_J$ such that $a \le a' z$ and $b' z \le b$. 

\begin{prop}
The maps 
\begin{gather*} 
f: Q'_J \to \Om_J, \quad (a, b, c) \mapsto (a, c b) \\
g: \Om_J \to Q_J, \quad (a, b) \mapsto (\min(b W_J), a b \i \min(b W_J)) \\
h: Q_J \to Q'_J, \quad (x, y) \mapsto (\max(y W_J), y \i \max(y W_J), x)
\end{gather*}
give order-preserving bijections between the posets $(Q'_J, \le)$, $(\Om_J, \le)$ and $(Q_J, \preceq)$. 
\end{prop}

\begin{proof}
For $(a, b, c) \in Q'_J$, $f(a, b, c)=(a, c b)$, $g \circ f(a, b, c)=(c, a b \i)$ and $h \circ g \circ f(a, b, c)=(a, b, c)$. Similarly, $g \circ f \circ h$ is an identity map on $Q_J$ and $f \circ h \circ g$ is an identity map on $\Om_J$. Thus $f, g, h$ are all bijective. 

Now it suffices to show that $f, g, h$ preserve the partial orders. 

Let $(a, b, c), (a', b', c') \in Q'_J$ with $(a', b', c') \le (a, b, c)$. Then there exists $u'_1 \in W_J$ such that $a b \i \le a' (b') \i u'_1 \le c' u'_1 \le c$.  Thus $a = ab\i * b \leq a'(b') \i u'_1 * b$.  In other words, there exists $v \le b$ such that $a \le a'(b') \i u'_1v$.  Let $z = (b') \i u'_1v$.  Then $c'b'z = c' u_1'v \le cv \le cb$ since $c \in W^J$.  Hence $(a', c' b') \le (a, c b)$ in $\Om_J$. 

Let $(a, b), (a', b') \in \Om_J$ with $(a', b') \le (a, b)$. Then there exists $z \in W_J$ with $a \le a' z$ and $b' z \le b$. We assume that $g(a, b)=(x, y)$ and $g(a', b')=(x', y')$. Then there exists $u, u' \in W_J$ such that $x=b u$, $y=a u$, $x'=b' u'$ and $y'=a' u'$. Since $b' z \le b$, we have that $b' z \tril u \le b \tril u \le b u=x$. Hence there exists $v \le u$ such that $x' (u') \i z v=b' z v \le x$. 

Since $a \in W^J_{\max}$, $$y=a u \le a v=a \tril v \le a' z \tril v \le a' z v=y' (u') \i z v.$$ Thus $(x', y') \preceq (x, y)$ in $Q_J$. 

Now let $(x, y), (x', y') \in Q_J$ with $(x', y') \preceq (x, y)$. Then there exists $u \in W_J$ such that $x' u \le x$ and $y' u \ge y$. So we have that $y' \ast u \ge y' u \ge y$. In other words, there exists $v \le u$ with $y' \ast u=y' v$ and $\ell(y' v)=\ell(y')+\ell(v)$. Since $x' \in W^J$, we also have that $y' v \le x' v \le x' u \le x$. 

We assume that $h(x, y)=(a, b, x)$ and $h(x', y')=(a', b', x')$. Then $y=a b \i$ and $y'=a' (b') \i$. We have that $$a b \i=y \le y' v=a' (b') \i v \le x' v \le x.$$ Since $a' \in W^J_{\max}$, $\ell(a' (b') \i v)=\ell(a')-\ell((b') \i v)$ and $\ell(y' v)=\ell(y')+\ell(v)=\ell(a')-\ell(b')+\ell(v)$. So $\ell((b') \i v)+\ell(v)=\ell(b')$. Thus $(a', b', x') \le (a, b, x)$ in $Q'_J$.
\end{proof}

\end{document}